\documentclass[11pt,a4paper]{amsart}
\usepackage{amssymb,amsmath,amsthm}

\usepackage[pdftex]{graphicx}

\usepackage[colorlinks=true, pdfstartview=FitV, linkcolor=blue, citecolor=blue, urlcolor=blue,pagebackref=false]{hyperref}

\usepackage{esint}
\usepackage{mathrsfs}

\usepackage{mathtools}
\usepackage{tikz}
\usetikzlibrary{positioning}
\usepackage{xcolor}
\usepackage{graphicx}
\usepackage{float}

\usepackage{times,latexsym,color}

\newcommand{\BOX}{\ensuremath\Box}

\newtheorem{theorem}{Theorem}
\newtheorem*{theorem*}{Theorem}
\newtheorem{pro}{Proposition}
\newtheorem{lemma}{Lemma}

\theoremstyle{remark}
\newtheorem{remark}{Remark}

\theoremstyle{definition}

\newcommand{\ben}{\begin{eqnarray}}
\newcommand{\een}{\end{eqnarray}}
\newcommand{\beno}{\begin{eqnarray*}}
\newcommand{\eeno}{\end{eqnarray*}}

\DeclareMathOperator*{\esssup}{ess\,sup}

\def\XXint#1#2#3{{\setbox0=\hbox{$#1{#2#3}{\int}$ }
\vcenter{\hbox{$#2#3$ }}\kern-.6\wd0}}

\definecolor{darkgreen}{rgb}{0,0.5,0}
\definecolor{darkblue}{rgb}{0,0,0.7}
\definecolor{darkred}{rgb}{0.9,0.1,0.1}
\definecolor{lightblue}{rgb}{0,0.51,1}

\begin{document}

\title{{Estimates of the singular set for the Navier-Stokes equations with supercritical assumptions on the pressure}}

\author[T. Barker]{Tobias Barker}
\address[T. Barker]{Department of Mathematical Sciences, University of Bath, Bath BA2 7AY. UK}
\email[T. Barker]{tobiasbarker5@gmail.com}
\author[W. Wang]{Wendong Wang}
\address[W. Wang]{School of Mathematical Sciences, Dalian University of Technology, Dalian, 116024, P. R. China.}
\email[W. Wang]{wendong@dlut.edu.cn}

\keywords{}
\subjclass[2010]{}
\date{\today}

\maketitle

\noindent {\bf Abstract}
In this paper, we investigate systematically the supercritical conditions on the pressure $\pi$ associated to a Navier-Stokes solution $v$ (in three-dimensions), which ensure a reduction in the Hausdorff dimension of the singular set at a first potential blow-up time.
As a consequence, we show that if the pressure $\pi$ satisfies the endpoint scale invariant conditions
$$\pi\in L^{r,\infty}_{t}L^{s,\infty}_{x}\quad\textrm{with}\,\,\tfrac{2}{r}+\tfrac{3}{s}=2\,\,\textrm{and}\,\,r\in (1,\infty),$$
then the Hausdorff dimension of the singular set at a first potential blow-up time is arbitrarily small. This hinges on two ingredients: (i) the proof of a higher integrability result for the Navier-Stokes equations with certain supercritical assumptions on $\pi$ and (ii) the establishment of a convenient $\varepsilon$- regularity criterion involving space-time integrals of
$$|\nabla v|^2|v|^{q-2}\,\,\,\textrm{with}\,\,q\in (2,3). $$
The second ingredient requires a modification of ideas in Ladyzhenskaya and Seregin's paper \cite{ladyzhenskayaseregin}, which build upon ideas in Lin \cite{lin}, as well as Caffarelli, Kohn and Nirenberg \cite{CKN}.

\section{Introduction}
In this paper, we consider the three-dimensional incompressible Navier-Stokes equations
\begin{equation}\label{NSintro}
\partial_{t}v-\Delta v+v\cdot\nabla v+\nabla \pi=0,\,\,\,\nabla\cdot v=0,\,\,\,v(\cdot,0)=v_{0}(x)\,\,\,\textrm{in}\,\,\,\mathbb{R}^3\times (0,T).
\end{equation}
Here, $T\in (0,\infty]$.

The rigorous existence theory for the Navier-Stokes equations was pioneered by Leray in \cite{Le}. In particular, Leray showed that for any square-integrable solenodial initial data $v_{0}(x)$ there exists at least one associated global-in-time \textit{weak Leray-Hopf solution} $v$. The question of global-in-time regularity of weak Leray-Hopf solutions remains open in three dimensions.

For weak Leray-Hopf solutions \textit{conditional} regularity results (in terms of the velocity field) were also given in Leray's paper \cite{Le}. In particular, \cite{Le} implies that for any $q\in (3,\infty]$ there exists a $C_{q}>0$ such that if $v$ is a weak Leray-Hopf solution (assumed to be smooth on $\mathbb{R}^3\times (0,T)$) then
\begin{equation}\label{Lerayssmoothness}
\|v(\cdot,t)\|_{L^{q}(\mathbb{R}^3)}\leq \frac{C_{q}}{(T-t)^{\frac{1}{2}(1-\frac{3}{q})}}\,\,\,\textrm{for}\,\,\,\textrm{some}\,\,\,\,t\in (0,T)\Rightarrow\quad v\,\,\textrm{can}\,\,\textrm{be}\,\,\textrm{smoothly}\,\,\textrm{extended}\,\,\textrm{past}\,\,T.
\end{equation}
Almost 70 years later, Escauriaza, Seregin and \v{S}ver\'{a}k's celebrated paper \cite{ESS} treated the case $q=3$. In particular, it was shown in \cite{ESS} that
\begin{equation}\label{ESSsmoothness}
\|v\|_{L^{\infty}_{t}(0,T; L^{3}(\mathbb{R}^3))}<\infty\Rightarrow\quad v\,\,\textrm{can}\,\,\textrm{be}\,\,\textrm{smoothly}\,\,\textrm{extended}\,\,\textrm{past}\,\,T.
\end{equation}
There have been many extensions of \eqref{ESSsmoothness}. See, for example, \cite{sereginL3limit}, \cite{albritton} and
\cite{albrittonbarker}.

Notice that the assumptions \eqref{Lerayssmoothness} and \eqref{ESSsmoothness} are invariant with respect to the scaling symmetry of the Navier-Stokes equations
\begin{equation}\label{eqrescaling}
(v_{\lambda}(x,t), \pi_{\lambda}(x,t), v_{0\lambda}(x))=(\lambda v(\lambda x,\lambda^2 t),\lambda^2 \pi(\lambda x,\lambda^2 t), \lambda v_{0}(\lambda x) )\,\,\textrm{with}\,\,\,\lambda\in (0,\infty).
\end{equation}
This paper is concerned with conditional results for the three-dimensional Navier-Stokes equations \textit{in terms of the pressure}, which are less well understood compared to regularity results formulated in terms of the velocity.

To the best of our knowledge, regularity criteria involving the pressure were first considered in Kaniel's paper \cite{kaniel}. Kaniel considered a weak solution $v:\Omega\times (0,\infty)\rightarrow \mathbb{R}^3$ ($\Omega$ is a smooth bounded domain and $v$ satisfies the Dirichlet boundary condition on $\partial\Omega$) with smooth initial conditions. Furthermore, Kaniel showed that if the associated pressure $\pi$ satisfies
\begin{equation}\label{kanielcondition}
\|\pi\|_{L^{\infty}_{t}(0,T; L^{q}(\Omega))}<\infty\quad\textrm{with}\,\,q>\frac{12}{5}
\end{equation}
then $v$ can be smoothly extended past $T$. Note that the assumed bound on the pressure in \eqref{kanielcondition} is not invariant with respect to the rescaling \eqref{eqrescaling} and is \textit{subcritical}\footnote{We say a quantity $F(v,\pi)\in [0,\infty)$ is \textit{subcritical} if, for the rescaling \eqref{eqrescaling}, there exists $\alpha>0$ such that $F(v_{\lambda},\pi_{\lambda})=\lambda^{\alpha} F(v,\pi)$ for all $\lambda>0$.}. Kaniel's result paved the way for other regularity criteria for the Navier-Stokes equations in terms of subcritical norms of the pressure. See, for example, \cite{berselli}, \cite{bdv} and \cite{chaelee}.

Conditional regularity results in terms of \textit{scale-invariant}\footnote{We say a quantity $F(v,\pi)\in [0,\infty)$ is \textit{scale-invariant}, with respect to the rescaling \eqref{eqrescaling}, if $F(v_\lambda, \pi_{\lambda})=F(v,\pi)$ for all $\lambda>0$.} norms of the pressure were pioneered by Berselli and Galdi's seminal paper \cite{berselligaldi}. In \cite{berselligaldi}, Berselli and Galdi considered a weak Leray-Hopf solution $v:\mathbb{R}^n\times (0,\infty)\rightarrow\mathbb{R}^n$, with sufficiently smooth initial data. They showed that if the associated pressure $\pi$ satisfies the scale-invariant bounds
\begin{align}\label{pressurescaleinvariant}
\begin{split}
&\|\pi\|_{L^r_{t}(0,T; L^{s}(\mathbb{R}^n))}<\infty\,\,\,\textrm{with}\,\,\,\frac{2}{r}+\frac{n}{s}=2\,\,\,\textrm{and}\,\,\,s>\frac{n}{2},\\
\textrm{or}\,\,\textrm{if}\,\,&\|\pi\|_{L^{\infty}_{t}(0,T; L^{\frac{n}{2}}(\mathbb{R}^n))}\,\,\textrm{is}\,\,\textrm{small}\,\,\textrm{enough},
\end{split}
\end{align}
then $v\in C^{\infty}((0,T]\times \mathbb{R}^n).$ Galdi and Berselli's influential work paved the way for many regularity criteria in terms of scale-invariant norms involving the pressure. Whilst it is not possible to list such works exhaustively, we refer the reader to \cite{zhou06proc}, \cite{zhou06math}, \cite{kanglee}, \cite{struwe}, \cite{caifangzhai}, \cite{suzucki1}, \cite{suzucki2} and \cite{ji}, for example. Let us mention that there are many interesting regularity criteria involving the pressure that are of a different nature to the aforementioned literature. We point out \cite{neustupanecas}, \cite{sereginsverakpressure}, \cite{caotiti}, \cite{constantin} and \cite{tranyu}.

Despite a substantial number of results concerning regularity criterion in terms of the pressure, there remains a ``scaling gap'' between knowledge of the pressure for any weak Leray-Hopf solution and conditions on the pressure sufficient to ensure regularity. In particular, for a weak Leray-Hopf solution $v:\mathbb{R}^3\times (0,\infty)\rightarrow \mathbb{R}^3$, it is known that the associated pressure $\pi$ belongs to $$L^{a}_{t}L^{b}_{x}(\mathbb{R}^3\times (0,\infty))\quad\textrm{with}\quad\frac{2}{a}+\frac{3}{b}=3\quad\textrm{and}\quad1\leq a<\infty.$$
Even more surprising is that it remains unclear if regularity of the Navier-Stokes equations holds when assuming the pressure analogue of Escauriaza, Seregin and \v{S}ver\'{a}k's condition \eqref{ESSsmoothness} or other scale-invariant endpoint cases. In particular, the following is open.
\begin{itemize}
\item[]\textbf{Open Problem.} Suppose $v:\mathbb{R}^3\times (0,\infty)\rightarrow\mathbb{R}^3$ is a weak Leray-Hopf solution, with sufficiently smooth initial data. Fix $s\in [\frac{3}{2},\infty)$ and $r\in (1,\infty]$ such that $\frac{2}{r}+\tfrac{3}{s}=2$.
$$\textrm{Does}\,\,\,\|\pi\|_{L^{r,\infty}(0,T; L^{s}(\mathbb{R}^3))}<\infty \Rightarrow v\in C^{\infty}((0,T]\times \mathbb{R}^3)? $$
\end{itemize}
Here, $L^{r,\infty}(\mathbb{R})$ is the Lorentz space, which is slightly larger\footnote{For example, $|t|^{-\frac{1}{r}}\in L^{r,\infty}(\mathbb{R})\setminus L^{r}(\mathbb{R})$. Throughout this paper, we use the convention that $L^{\infty}=L^{\infty,\infty}$. } than $L^{r}(\mathbb{R})$ and is defined in `2.Preliminaries'.

This paper is motivated by these open problems. Specifically, we are motivated by the following question.
\begin{itemize}
\item[]\textbf{(Q)} \textit{Consider a weak Leray-Hopf solution of the three-dimensional Navier-Stokes equations. What assumptions on the pressure ensure a reduction in the Hausdorff dimension of the singular set at a first potential blow-up time?}
\end{itemize}
\subsection{Main Results}
Note that if $v:\mathbb{R}^3\times (0,\infty)\rightarrow\mathbb{R}^3$ is a weak Leray-Hopf solution, which first loses smoothness at $T^*>0$, then Caffarelli, Kohn and Nirenberg's seminal paper \cite{CKN} implies that
$$\mathcal{H}^{1}(\sigma)=0. $$
Without extra assumptions, it is unknown whether the Hausdorff dimension can be reduced below 1 beyond logarithmic factors (see \cite{choelewis}). Our main results show that the additional  supercritical\footnote{We say a quantity $F(v,\pi)\in [0,\infty)$ is \textit{supercritical} if, for the rescaling \eqref{eqrescaling}, there exists $\alpha>0$ such that $F(v_{\lambda},\pi_{\lambda})=\lambda^{-\alpha} F(v,\pi)$ for all $\lambda>0$.} pressure assumption \eqref{pressuresupercritmaintheo} or the endpoint scale-invariant pressure assumption \eqref{typeIpresgenmaintheo2} give an improvement in the known (unconditional) dimension of the singular set.
Let us now state our first main result.
\begin{theorem}\label{hausdorffdimreduce}
Let $\delta\in (0,\frac12)$, $\gamma\in (0,\frac12)$, $p\in (\frac{3}{2+\gamma},\frac{3}{\gamma})$ and $r\in (1,\infty)$. Suppose further that $(p,r,\delta,\gamma)$ satisfy the following relations.
\begin{itemize}
\item[1.]$\frac{2}{r}+\tfrac{3}{p}=2+\gamma$.

\item[2.] For $r\geq 2$, we take $\frac{3\gamma}{2+2\gamma}<\delta<\frac{1}{2}$ and $g(\delta,\gamma)<p,$
where \beno
g(\delta,\gamma):= \frac{3\delta}{2\delta-(1-\delta)\gamma}.
\eeno
\item[3.] For $1<r<2$, we  take $\frac{3\gamma}{2+2\gamma}<\delta<\frac{1}{2}$ and
$\frac{3}{1+\gamma}< p\leq\frac{2\delta}{\gamma}$.
\end{itemize}

Let $v$ be a weak Leray-Hopf solution to the Navier-Stokes equations on $\mathbb{R}^3\times (0,\infty)$. Assume that $v$ first blows-up at $T^*>0$, namely
$$v\in L^{\infty}_{loc}([0,T^*); L^{\infty}(\mathbb{R}^3))\,\,\,\textrm{and}\,\,\,\lim_{t\uparrow T^*}\|v(\cdot,t)\|_{L^{\infty}(\mathbb{R}^3)}=\infty. $$
Assume that the pressure $\pi$ associated to $v$ satisfies
\begin{equation}\label{pressuresupercritmaintheo}
\pi\in L^{r}(0,T^*; L^{p}(\mathbb{R}^3)),
\end{equation}
where $(p,r,\delta,\gamma)$ are as specified in the first paragraph.
Let\footnote{Note that $(x,T^*)$ is a singular point of $v$ if $v\notin L^{\infty}(B(x,R)\times (T^*-R^2,T^*))$ for all sufficiently small $R<T^*$.}
\begin{equation}\label{sigmadefhausdorff}
\sigma:=\{x: (x,T^*)\,\,\,\textrm{is}\,\,\,\textrm{a}\,\,\,\textrm{singular}\,\,\,\textrm{point}\,\,\,\textrm{of}\,\,\,v\}.
\end{equation}
Then the above assumptions imply that
$$\mathcal{H}^{2\delta}(\sigma)=0.$$
Here, $\mathcal{H}^{2\delta}$ denotes the Hausdorff measure of dimension $2\delta.$
\end{theorem}
$$\pi\in L^{r}(0,T^*; L^{p}(\mathbb{R}^3))\quad\textrm{or}\quad v\in L^{2r}(0,T^*; L^{2p}(\mathbb{R}^3)). $$
\begin{remark}
The indices of the above theorem are represented in Figure 1. In particular, when the indices are in the shaded region, the Hausdorff dimension of $\sigma$ is strictly less than 1 due to Theorem \ref{hausdorffdimreduce}. The line $L_2$ has the formula $\frac1r+\frac3p=2$ and the line $L_1$ has the formula $\frac1r+\frac1p=1$ . Note that in case 2. of Theorem \ref{hausdorffdimreduce} we have $\frac3p=2+\gamma-\frac2r<2-(\frac{1}{\delta}-1)\gamma$, which implies $\gamma<\frac1r$. Hence in case 2., $\frac2r+\frac3p=2+\gamma<2+\frac1r$, which is $L_2$. Similarly, for case 3. of Theorem \ref{hausdorffdimreduce}, we have
$\frac3p=2+\gamma-\frac2r>\frac{3\gamma}{2\delta}>3\gamma$, which implies $\gamma<1-\frac1r$. Hence, in case 3., $\frac2r+\frac3p=2+\gamma<3-\frac1r,$ which is $L_1$.
\end{remark}
\begin{figure}
		\begin{minipage}[t]{0.3\textwidth}
\centering
\begin{tikzpicture}[domain=0:8]
\draw[->](0,0)--(8.2,0)node[below]{$\frac{1}{p}$};
\draw[->](0,0)--(0,8.2)node[above]{$\frac{1}{r}$};
\draw  [black](4,0)--(0,6)   (0,6)node[left]{$1$}-- node [color=black!70,pos=0.25,below,sloped]{$\frac2r+\frac3p=2$} (4,0)node[below]{$\frac23$};
\draw  [black](6,0)--(2,6)   (2,6)node[left]{$(\frac13,1)$}-- node [color=black!70,pos=0.25,above,sloped]{$\frac2r+\frac3p=3$} (6,0)node[below]{$1$};
\draw  [blue,dashed](6,0)--(0,6) -- node[color=black!70,pos=0.25,above,sloped]{$L_1$}  (3,3)node[right]{$(\frac12, \frac12)$};
\draw  [blue,dashed](3,3)--(4,0) (4,1/2)node[color=black!70,right]{$L_2$};
\draw  [blue,dashed](3,3)--(2,6);
\draw  [blue,dashed](0,3)--(3,3)   (0,3)node[left]{$\frac12$};
\draw  [blue,dashed](2,0)--(2,3)   (2,0)node[below]{$\frac13$};
\filldraw [gray] (4,0)--(0,6)--(3,3);
\draw[->](2.5,3)--(6,5)node[right]{$\mathcal{H}^{2\delta}(\sigma)=0,\delta\in (0,\frac12)$};
\end{tikzpicture}
\caption{}
\end{minipage}
\end{figure}

Let us state our second main result, which treats endpoint cases for scale-invariant norms of the pressure.
\begin{theorem}\label{hausdorffdimreducegen2}
Fix $s\in (\frac{3}{2},\infty)$ and $r\in (1,\infty)$ such that $\frac{2}{r}+\tfrac{3}{s}=2$. Let $v$ be a weak Leray-Hopf solution to the Navier-Stokes equations on $\mathbb{R}^3\times (0,\infty)$. Assume that $v$ first blows-up at $T^*>0$, namely
$$v\in L^{\infty}_{loc}([0,T^*); L^{\infty}(\mathbb{R}^3))\,\,\,\textrm{and}\,\,\,\lim_{t\uparrow T^*}\|v(\cdot,t)\|_{L^{\infty}(\mathbb{R}^3)}=\infty. $$
Assume that the pressure $\pi$ associated to $v$ satisfies the scale-invariant bound
\begin{equation}\label{typeIpresgenmaintheo2}
\pi\in {L^{r,\infty}(0,T^*; L^{s,\infty}(\mathbb{R}^3))}.
\end{equation}
Let
\begin{equation}\label{sigmadefhausdorffgen2}
\sigma:=\{x: (x,T^*)\,\,\,\textrm{is}\,\,\,\textrm{a}\,\,\,\textrm{singular}\,\,\,\textrm{point}\,\,\,\textrm{of}\,\,\,v\}.
\end{equation}
Then the above assumptions imply that
$$\mathcal{H}^{2\delta}(\sigma)=0,$$
for any $\delta\in (0,\frac{1}{2})$.
Here, $\mathcal{H}^{2\delta}$ denotes the Hausdorff measure of dimension $2\delta.$
\end{theorem}

In order to prove Theorems \ref{hausdorffdimreduce}-\ref{hausdorffdimreducegen2}, we need to prove a generalization of the higher integrability result in \cite{barkerhigherinteg} by means of the following theorem.
\begin{theorem}\label{thm:higherinteggen}
Let $\delta\in (0,\frac12)$, $\gamma\in (0,\frac12)$, $p\in (\frac{3}{2+\gamma},\frac{3}{\gamma})$ and $r\in (1,\infty)$. Suppose further that $(p,r,\delta,\gamma)$ satisfy the following relations.
\begin{itemize}
\item[1.]$\frac{2}{r}+\tfrac{3}{p}=2+\gamma$.

\item[2.] For $r\geq 2$, we take $\frac{3\gamma}{2+2\gamma}<\delta<\frac{1}{2}$ and $g(\delta,\gamma)<p,$
where \beno
g(\delta,\gamma):=\frac{3\delta}{2\delta-(1-\delta)\gamma}.
\eeno
\item[3.] For $1<r<2$, we  take $\frac{3\gamma}{2+2\gamma}<\delta<\frac{1}{2}$ and
$\frac{3}{1+\gamma}< p\leq\frac{2\delta}{\gamma}$.
\end{itemize}

Let $v$ be a weak Leray-Hopf solution to the Navier-Stokes equations on $\mathbb{R}^3\times (0,\infty)$. Assume that $v$ first blows-up at $T^*>0$, namely
\begin{equation}\label{vbounded2}
v\in L^{\infty}_{loc}([0,T^*); L^{\infty}(\mathbb{R}^3))\,\,\,\textrm{and}\,\,\,\lim_{t\uparrow T^*}\|v(\cdot,t)\|_{L^{\infty}(\mathbb{R}^3)}=\infty.
\end{equation}
Assume that the pressure $\pi$ associated to $v$ satisfies
\begin{equation}\label{pressuresupercrit}
\pi\in L^{r}(0,T^*; L^{p}(\mathbb{R}^3)),
\end{equation}
where $(p,r,\delta,\gamma)$ are as specified in the first paragraph.

Then the above assumptions imply that we have higher integrability \textbf{up to $T^*$}. Namely, for all $t_{1}\in (0,T^*)$ we have
\begin{equation}\label{higherintegrable2}
\||v|^{\frac{q}{2}}\|_{L^{\infty}((t_1,T^*); L^{2}(\mathbb{R}^3))}^2+\int\limits_{t_1}^{T^*}\int\limits_{\mathbb{R}^3} |\nabla v|^2|v|^{q-2} dxds+\int\limits_{t_1}^{T^*}\int\limits_{\mathbb{R}^3} |\nabla (|v|^{\frac{q}{2}})|^2 dxds<\infty\,\,\,\textrm{with}\,\,\,q:=3-2\delta.
\end{equation}

\end{theorem}
Note that in the above Theorem, we always have $\delta>\frac{3\gamma}{2+2\gamma}$. This means that the scaling of the velocity estimates in \eqref{higherintegrable2}, with respect to the Navier-Stokes scaling symmetry, is worse than that of the corresponding pressure assumption \eqref{pressuresupercrit}. Indeed, their equivalence would require $\delta=\frac{3\gamma}{2\gamma+4}.$ This discrepancy in scaling is caused by the supercritical nature of the pressure assumption \eqref{pressuresupercrit}, together with the fact that $\pi$ has quadratic dependence on $v$.
\subsection{Strategy For the Proof of Higher Integrability (Theorem \ref{thm:higherinteggen})}
As with many of the aforementioned papers, a key idea in proving Theorem \ref{thm:higherinteggen} is to test the Navier-Stokes equations with $v|v|^{1-2\delta}$ (for $\delta\in (0,\frac{1}{2})$. For $t,t_1\in (0,T^*)$, this yields
\ben\label{eq:estimate energy-generalintro}
&&E_{\delta}(v)(t):=\sup_{t_1\leq s\leq t}\int_{\mathbb{R}^3} |v(x,s)|^{3-2\delta}dx+\int_{t_1}^t\int_{\mathbb{R}^3}|\nabla( |v|^{\frac32-\delta})|^2dxds\nonumber\\
&\leq &C(\delta)\left|\int_{t_1}^t\int_{\mathbb{R}^3}|\pi||\nabla( |v|^{\frac32-\delta})| |v|^{\frac12-\delta}dxds\right|+E_{\delta}(v)(t_1).
\een
Our goal to prove Theorem \ref{thm:higherinteggen} is to use the assumption $\pi\in L^{r}_{t}L^{p}_{x}(\mathbb{R}^3\times (0,T^*))$ to estimate \eqref{eq:estimate energy-generalintro}. Specifically, we aim to obtain that for some $\theta\in [\frac{1}{2},1]$ and $\phi\in (0,1]$ that the following estimate holds true uniformly for all $t\in (t_1,T^*)$. Namely,
\begin{equation}\label{energyestintro}
E_{\delta}(v)(t)\leq C(\delta)\|\pi\|^{\theta}_{L^{r}_{t}L^{p}_{x}(\mathbb{R}^3\times (t_1,T^*))}(E_{\delta}(v)(t))^{\phi}+E_{\delta}(v)(t_1).
\end{equation}
Provided that $r<\infty$ (with the possibility that $t_1$ must be chosen sufficiently close to $T^*$), this then allows us to conclude that $E_{\delta}(v)(T^*)<\infty$. This gives the desired higher integrability.
\subsection{Strategy For the Reduction in the Hausdorff Dimension of the Singular Set}
Recall that in Caffarelli, Kohn and Nirenberg's paper \cite{CKN} $\varepsilon$-regularity criteria, formulated in terms of the quantities
$$r^{-1}\int\limits_{t-\frac{7}{8}r^2}^{t+\frac{1}{8}r^2}\int\limits_{B(x_0,r)} |\nabla v|^2 dxds $$
(see Proposition 2 in \cite{CKN}), are used in conjunction with a covering argument to obtain that the one-dimensional parabolic Hausdorff dimension of the space-time singular set is zero.
Notice that from the dimensional analysis perspective of \cite{CKN}, the space-time integral of $|\nabla v|^2$ has dimension 1. Since $|\nabla v|^2$ is controlled for certain classes of solutions of the Navier-Stokes equations, this explains why it is heuristically reasonable that the one-dimensional parabolic Hausdorff dimension of the space-time singular set is zero.

Adopting the same dimensional analysis perspective of \cite{CKN}, we see that for $q\in (2,3)$
$|\nabla v|^2|v|^{q-2}$ has dimension $3-q$. Due to Theorem  \ref{thm:higherinteggen}, one should heuristically expect that, under the assumptions of Theorems \ref{hausdorffdimreduce}-\ref{hausdorffdimreducegen2}, the dimension of the singular set is reduced.
In order to make these heuristics rigorous, it is necessary to formulate an $\varepsilon$-regularity criteria involving the space-time integral of $$|\nabla v|^2|v|^{q-2}. $$ We accomplish this by means of the proposition below.\begin{pro}\label{Bqsmallreg}
For every $q\in (2,3)$, there exists $\varepsilon_{q}\in (0,1)$ such that the following holds true.

Suppose that $(v,\pi)$ is a suitable weak solution\footnote{See `2.3 Solution Classes of the Navier-Stokes Equations' for a definition of `suitable weak solutions'.} to the Navier-Stokes equations in $Q(0,1):=B(0,1)\times (-1,0)$.
Furthermore, suppose that
\begin{equation}\label{Bqsmall}
\sup_{0<r<1}\Big(\frac{1}{r^{3-q}}\int\limits_{-r^2}^{0}\int\limits_{B(0,r)} |\nabla v|^2|v|^{q-2} dxds\Big)<\varepsilon_{q}.
\end{equation}
Then the above assumptions imply that $(x,t)=(0,0)$ is a regular point of $v$. In particular, there exists an $R\in (0,1)$ such that
$v\in L^{\infty}(B(0,R)\times (-R^2,0)).$
\end{pro}
\subsection{Final Remarks}
The methods of this paper cannot be used to treat the case when the pressure satisfies the scale-invariant assumption
\begin{equation}\label{pressurelorentznotLebesgue}
\|\pi\|_{L^{\infty}(0,T^*; L^{\frac{3}{2},\infty}(\mathbb{R}^3))}\leq M^2.
\end{equation}

In particular, we cannot obtain higher integrability by using interpolation arguments and  applying Theorem \ref{thm:higherinteggen}.

A different argument in \cite{barkerhigherinteg}, which uses a `small exponent', gives higher integrability up to $T^*$ when the pressure satisfies \eqref{pressurelorentznotLebesgue}. Specifically, \cite{barkerhigherinteg} shows that there exists a positive universal constant $C_{univ}^{(0)}$ such that for all $t_1\in (0,T^*)$:
\begin{equation}\label{higherintegfirstblowup}
\||v|^{\frac{q}{2}}\|_{L^{\infty}_{t}(t_1,T^*; L^{2}(\mathbb{R}^3))}^2+\int\limits_{t_1}^{T^*}\int\limits_{\mathbb{R}^3} |\nabla v|^2|v|^{q-2} dxds+\int\limits_{t_1}^{T^*}\int\limits_{\mathbb{R}^3} |\nabla |v|^{\frac{q}{2}}|^2 dxds<\infty\quad\textrm{with}\,\,q=2+\frac{C_{univ}^{(0)}}{M}.
\end{equation}
One can then argue as in the proof of Theorem \ref{hausdorffdimreduce} to show that the singular set $\sigma$ at the first blow-up time $T^*$ satisfies
$$\mathcal{H}^{1-\frac{C_{univ}^{(0)}}{M}}(\sigma),$$
under the assumption \eqref{pressurelorentznotLebesgue}.

Let us remark that after this paper, we still do not know if $$\pi\in L^{\infty}(0,T^*;L^{\frac{3}{2}-}(\mathbb{R}^3))$$ implies a reduction in the Hausdorff dimension of the singular set $\sigma$ at the first blow-up time $T^*$. Perhaps this indicates that the notion of a ``scaling gap'' may be too blunt when trying to understand the singular set. In order to understand how the singular set might transition between weak Leray-Hopf solutions and smooth solutions, it is arguably more appropriate to use quantities that capture oscillatory behavior. Such an example is given by Proposition \ref{Bqsmallreg}, on the level of the velocity field.
 \begin{section}{Preliminaries}
\subsection{General Notation}

Throughout this paper we adopt the Einstein summation convention. For arbitrary vectors $a=(a_{i}),\,b=(b_{i})$ in $\mathbb{R}^{n}$ and for arbitrary matrices $F=(F_{ij}),\,G=(G_{ij})$ in $\mathbb{M}^{n}$ we put
 $$a\cdot b=a_{i}b_{i},\,|a|=\sqrt{a\cdot a},$$
 $$a\otimes b=(a_{i}b_{j})\in \mathbb{M}^{n},$$
 $$FG=(F_{ik}G_{kj})\in \mathbb{M}^{n}\!,\,\,F^{T}=(F_{ji})\in \mathbb{M}^{n}\!,$$
 $$F:G=
 F_{ij}G_{ij}\,\,\,\textrm{and}
 \,\,\,|F|=\sqrt{F:F}.$$
For $x_0\in\mathbb{R}^n$ and $R>0$, we define the ball
\begin{equation}\label{balldef}
B(x_0,R):=\{x: |x-x_0|<R\}.
\end{equation}
For $z_0=(x_0,t_0)\in \mathbb{R}^n\times\mathbb{R}$ and $R>0$, we denote the parabolic cylinder by
\begin{equation}\label{paraboliccylinderdef}
Q(z_0,R):=\{(x,t): |x-x_0|< R,\,t\in (t_0-R^2,t_0)\}.
\end{equation}
For $S\subset \mathbb{R}^n$, $\varepsilon>0$ and $\lambda\in (0,\infty)$ we define
$$\mathcal{H}^{\lambda, \varepsilon}(S):=\inf\Big\{\sum_{i=1}^{\infty} (\textrm{diam}\,U_{i})^{\lambda}: S\subset \cup_{i=1}^{\infty} U_{i}\,\,\,\textrm{with}\,\,\,\textrm{diam}\,U_{i}\leq \varepsilon\Big\}. $$
We then define the $\lambda$-dimensional Hausdorff dimension of $S$ as
$$\mathcal{H}^{\lambda}(S):=\lim_{\varepsilon\rightarrow 0^+}\mathcal{H}^{\lambda, \varepsilon}(S). $$

  If $X$ is a Banach space with norm $\|\cdot\|_{X}$, then $L_{s}(a,b;X)$, with $a<b$ and $s\in[1,\infty)$,  will denote the usual Banach space of strongly measurable $X$-valued functions $f(t)$ on $(a,b)$ such that
$$\|f\|_{L^{s}(a,b;X)}:=\left(\int\limits_{a}^{b}\|f(t)\|_{X}^{s}dt\right)^{\frac{1}{s}}<+\infty.$$
The usual modification is made if $s=\infty$.
Sometimes we will denote $L^{p}(0,T; L^{q})$ by $L^{p}_{t}L^{q}_{x}$, $L^{p}(0,T; L^{q}_{x})$ or $L^{p}_{t}(0,T; L^q)$.

Let $C([a,b]; X)$ denote the space of continuous $X$ valued functions on $[a,b]$ with the usual norm. In addition, let $C_{w}([a,b]; X)$ denote the space of $X$ valued functions, which are continuous from $[a,b]$ to the weak topology of $X$.

\subsection{Lorentz Spaces}
Given a measurable subset $\Omega\subseteq\mathbb{R}^{n}$, let us define the Lorentz spaces.
For a measurable function $f:\Omega\rightarrow\mathbb{R}$ define:
\begin{equation}\label{defdistchapter2}
d_{f,\Omega}(\alpha):=\mu(\{x\in \Omega : |f(x)|>\alpha\}),
\end{equation}
where $\mu$ denotes the Lebesgue measure on $\mathbb{R}^n$.
 The Lorentz space $L^{p,q}(\Omega)$, with $p\in [1,\infty)$, $q\in [1,\infty]$, is the set of all measurable functions $g$ on $\Omega$ such that the quasinorm $\|g\|_{L^{p,q}(\Omega)}$ is finite. Here:

\begin{equation}\label{Lorentznormchapter2}
\|g\|_{L^{p,q}(\Omega)}:= \Big(p\int\limits_{0}^{\infty}\alpha^{q}d_{g,\Omega}(\alpha)^{\frac{q}{p}}\frac{d\alpha}{\alpha}\Big)^{\frac{1}{q}},
\end{equation}
\begin{equation}\label{Lorentznorminftychapter2}
\|g\|_{L^{p,\infty}(\Omega)}:= \sup_{\alpha>0}\alpha d_{g,\Omega}(\alpha)^{\frac{1}{p}}.
\end{equation}\\
Notice that if $p\in(1,\infty)$ and $q\in [1,\infty]$, there exists a norm, which is equivalent to the quasinorm defined above, for which $L^{p,q}(\Omega)$ is a Banach space.
For $p\in [1,\infty)$ and $1\leq q_{1}< q_{2}\leq \infty$, we have the following continuous embeddings
\begin{equation}\label{Lorentzcontinuousembeddingchapter2}
L^{p,q_1}(\Omega) \hookrightarrow  L^{p,q_2}(\Omega)
\end{equation}
and the inclusion is known to be strict.

We will make use of the following known interpolative inequality for Lorentz spaces. For a proof, see (for example) Proposition 1.1.14 of \cite{grafakos}.
\begin{pro}\label{interpolativepropertyLorentz}
Suppose that $1\leq p<r<q\leq\infty$ and
let $0<\theta<1$ be such that
$$\frac{1}{r}=\frac{\theta}{p}+\frac{1-\theta}{q}.$$
Then the assumption that $f\in L^{p,\infty}(\Omega)\cap L^{q,\infty}(\Omega)$ implies $f\in L^{r}(\Omega)$ with the estimate
\begin{equation}\label{interpolationestimate}
\|f\|_{L^{r}(\Omega)}\leq \Big(\frac{r}{r-p}+\frac{r}{q-r}\Big)^{\frac{1}{r}}\|f\|_{L^{p,\infty}(\Omega)}^{\theta}\|f\|_{L^{q,\infty}(\Omega)}^{1-\theta}.
\end{equation}
\end{pro}
\begin{subsection}{Solution Classes of the Navier-Stokes Equations}

 We say $v$ is a \emph{finite-energy solution} or a \emph{weak Leray-Hopf solution} to the Navier-Stokes equations on $(0,T)$ if $v\in C_{w}([0,T]; L^{2}_{\sigma}(\mathbb{R}^3))\cap L^{2}(0,T; \dot{H}^{1}(\mathbb{R}^3))$ and if $v$ satisfies the global energy inequality
$$\|v(\cdot,t)\|_{L^{2}(\mathbb{R}^3)}^2+2\int\limits_{0}^{t}\int\limits_{\mathbb{R}^3}|\nabla v|^2dxdt'\leq \|v(\cdot,0)\|_{L^{2}(\mathbb{R}^3)}^2$$
for all $t\in [0,T)$.

Let $\Omega\subseteq\mathbb{R}^3$. We say that $(v,\pi)$ is a \textit{suitable weak solution} to the Navier-Stokes equations
in $\Omega\times (T_{1},T)$ if it fulfills the properties described in \cite{gregory2014lecture} (Definition 6.1 p.133 in \cite{gregory2014lecture}).
\end{subsection}

\end{section}
\section{Proof of higher integrability (Theorem \ref{thm:higherinteggen})}
Note that the fact that $v$ is a weak Leray-Hopf solution satisfying \eqref{vbounded2} implies that $v$ and its associated pressure $\pi$ are sufficiently smooth on $\mathbb{R}^3\times (0,T^*)$. For each $t\in (0,T^*)$, $v(\cdot,t)$ belongs to every $L^{\alpha}(\mathbb{R}^3)$ with $\alpha \in [2,\infty]$. For each $t\in (0,T^*)$, $\pi(\cdot,t)$ belongs to every $L^{\beta}(\mathbb{R}^3)$ with $\beta \in (1,\infty)$.
Hence all calculations performed below can be rigorously justified.

For $\delta\in (0,\frac12)$, to be decided, testing the Navier-Stokes equations with $v|v|^{1-2\delta}$ yields
\ben\label{eq:estimate energy-general}
&&\sup_{t_1<s<t}\int_{\mathbb{R}^3} |v(x,s)|^{3-2\delta}dx+\int_{t_1}^t\int_{\mathbb{R}^3}|\nabla( |v|^{\frac32-\delta})|^2dxds\nonumber\\
&\leq &C(\delta)\left|\int_{t_1}^t\int_{\mathbb{R}^3}|\pi||\nabla( |v|^{\frac32-\delta})| |v|^{\frac12-\delta}dxds\right|+\int_{\mathbb{R}^3} |v(x,t_1)|^{3-2\delta}dx
\een

First of all,  H\"{o}lder and Calder\'{o}n-Zygmund inequalities yield that
\beno
I&=&\int_{t_1}^t\int_{\mathbb{R}^3}|\pi||\nabla( |v|^{\frac32-\delta})| |v|^{\frac12-\delta}dxds\\
&=& \int_{t_1}^t\int_{\mathbb{R}^3}|\pi|^{\theta}|\nabla( |v|^{\frac32-\delta})||\pi|^{1-\theta} |v|^{\frac12-\delta}dxds\\
&\leq &\int_{t_1}^t\|\pi\|_{L^p_x}^\theta \|\nabla( |v|^{\frac32-\delta})\|_{L^2_x} \||\pi|^{1-\theta}  |v|^{\frac12-\delta}\|_{L^\lambda_x}    ds\\
&\leq &\int_{t_1}^t\|\pi\|_{L^p_x}^\theta \|\nabla( |v|^{\frac32-\delta})\|_{L^2_x} \|\pi\|_{L_x^{\frac{\lambda}{2}( \frac52-2\theta-\delta)}}^{1-\theta}  \|v\|_{L^{\lambda( \frac52-2\theta-\delta)}_x}^{\frac12-\delta}   ds\\
&\leq &C_{\lambda,\delta,\theta}\int_{t_1}^t\|\pi\|_{L^p_x}^\theta \|\nabla( |v|^{\frac32-\delta})\|_{L^2_x}   \||v|^{\frac32-\delta}\|_{L^{\frac{\lambda( \frac52-2\theta-\delta)}{\frac32-\delta}}_x}^{\frac{\frac52-2\theta-\delta}{\frac32-\delta}}   ds\\
&\leq &C_{\lambda,\delta,\theta}\|\pi\|_{L_t^rL^p_x(\mathbb{R}^3\times (t_1,T^*))}^\theta \|\nabla( |v|^{\frac32-\delta})\|_{L^2_{t,x}(\mathbb{R}^3\times (t_1,t))}   \||v|^{\frac32-\delta}\|_{L^{\frac{\mu( \frac52-2\theta-\delta)}{\frac32-\delta}}_tL^{\frac{\lambda( \frac52-2\theta-\delta)}{\frac32-\delta}}_x (\mathbb{R}^3\times (t_1,t))}^{\frac{\frac52-2\theta-\delta}{\frac32-\delta}},
\eeno
where $\lambda,\mu\geq 1$ and $\theta\in [\frac12,1] $ are to be decided. Also,  we need the following conditions
\ben\label{eq:A}
\frac{\theta}{p}+\frac12+\frac{1}{\lambda}=1,
\een
\ben\label{eq:B}
\frac{\theta}{r}+\frac12+\frac{1}{\mu}=1,
\een
and for the Calder\'{o}n-Zygmund  estimates
\ben\label{eq:D-}
1<\frac{\lambda}{2}( \frac52-2\theta-\delta)
\een
Moreover, we also need the following range to close the energy estimate
\ben\label{eq:C}
{\frac{\frac52-2\theta-\delta}{\frac32-\delta}}\leq 1,
\een
\ben\label{eq:D}
(\frac32-\delta)\leq \frac{\lambda}{2}( \frac52-2\theta-\delta)\leq 3(\frac32-\delta)
\een
and
\ben\label{eq:E}
\frac2{\mu}+\frac3{\lambda}\geq \frac32\left(\frac{\frac52-2\theta-\delta}{\frac32-\delta}\right)
\een

First, for $\delta\in (0,\frac{1}{2})$ the condition of (\ref{eq:D-}) is included in (\ref{eq:D}) and (\ref{eq:C}) holds obviously for $\theta\geq \frac12$. By (\ref{eq:A}) and (\ref{eq:B}), we have
\ben\label{eq:theta}
\theta\in [\frac12, \min\{1,\frac{r}{2}, \frac{p}{2}\}]
\een
Second, due to $\frac2r+\frac3p=2+\gamma$, it follows from (\ref{eq:A}) and (\ref{eq:B}) that
\beno\label{eq:mu-lambda}
\frac2{\mu}+\frac3{\lambda}=\frac52-\theta(\frac2r+\frac3p)=\frac52-\theta(2+\gamma).
\eeno
This and  (\ref{eq:E}) yields that
\beno
\frac52-\theta(2+\gamma)\geq \frac32\left(\frac{\frac52-2\theta-\delta}{\frac32-\delta}\right),
\eeno
which implies that
\ben\label{eq:delta lower}
\theta\geq \frac{\delta}{[2\delta-\gamma(\frac32-\delta)]}
\een
and $\gamma< \frac{4\delta}{3-2\delta}$.
Especially, for $0<\gamma\leq \frac{2\delta}{3-2\delta}$ there holds
\beno\label{eq:gamma bound}
\frac12<\frac{\delta}{[2\delta-\gamma(\frac32-\delta)]}\leq 1.
\eeno
Thus the above conditions (\ref{eq:A})-(\ref{eq:delta lower}) except (\ref{eq:D}) imply that
\begin{equation}\label{eq:gamma bound1}
\frac12<\frac{\delta}{[2\delta-\gamma(\frac32-\delta)]}\leq\theta\leq   \min\{1,\frac{r}{2}, \frac{p}{2}\},\quad  0<\gamma\leq \frac{2\delta}{3-2\delta}.
\end{equation}

Finally, (\ref{eq:D})  and (\ref{eq:A}) imply
\ben\label{2sided}
(1-\frac{2\theta}{p})(\frac32-\delta)\leq \frac52-2\theta-\delta\leq (3- \frac{6\theta}{p} )(\frac32-\delta).
\een
The first part of \eqref{2sided}, together with (\ref{eq:gamma bound}) and $p>\frac{3}{2+\gamma}\geq \frac{3}{2}-\delta$, imply that
\ben\label{eq:D1}
\left(2-\frac2p(\frac32-\delta)\right)\frac{\delta}{[2\delta-\gamma(\frac32-\delta)]} \leq 1.
\een
Besides, the second one of (\ref{eq:D}) is
\ben\label{eq:D2}
\theta(9-6\delta-2p)\leq (2-2\delta)p
\een

Our main aim is to determine suitable $\delta\in(0,\frac12)$ and $\theta\in(\frac12,1]$ satisfying the conditions of \eqref{eq:gamma bound}, (\ref{eq:D1}) and (\ref{eq:D2}) for fixed $p$ and $r$, since $\gamma$ is decided by $p$ and $r$.

{\bf \underline{Case I: $r\geq 2$.}} At this time $\frac{3}{2+\gamma}<p\leq\frac{3}{1+\gamma}$. Moreover, (\ref{eq:D1}) holds if
\beno
\left(2-\frac23(1+\gamma)(\frac32-\delta)\right)\frac{\delta}{[2\delta-\gamma(\frac32-\delta)]} \leq 1.
\eeno
which is
\beno
0<\gamma\leq \frac{2\delta}{3-2\delta}.
\eeno
Hence  $\delta$ and $\theta$ are decided by (\ref{eq:gamma bound}) and (\ref{eq:D2}).

 In Case I below, we take any fixed $\delta\in (\frac{3\gamma}{2+2\gamma},\frac{1}{2})$ and $p\in (g(\delta,\gamma),\frac{3}{1+\gamma}].$ We show that (\ref{eq:gamma bound}) and (\ref{eq:D2}) can be satisfied with an appropriate $\theta$. In order to demonstrate this, we separately consider the cases $\frac52-\delta\leq p<\frac92-3\delta$ (Case I.1), $\frac{9-6\delta}{4-2\delta}\leq p<\frac52-\delta$ (Case I.2) and $g(\delta,\gamma)<p<\frac{9-6\delta}{4-2\delta}$ (Case I.3).

{\bf Case I.1: $\frac52-\delta\leq p<\frac92-3\delta$.}

First note that this case is non-empty since for $\gamma\in (0,\frac{1}{2})$ we have
$$\delta\geq \frac{3\gamma}{2+2\gamma}=\max\Big(\frac{5\gamma-1}{2+2\gamma},\frac{3\gamma}{2+2\gamma}\Big).$$
 In this case, we have
\beno
1< \frac{p}{2}\leq \frac{(2-2\delta)p}{9-6\delta-2p}
\eeno
and  (\ref{eq:D2}) holds. One can take $\theta=1$.

{\bf Case I.2: $\frac{9-6\delta}{4-2\delta}\leq p<\frac52-\delta$.}
Then  (\ref{eq:D2}) becomes
\beno
\frac{(2-2\delta)p}{9-6\delta-2p} \geq 1\geq \theta
\eeno
and $\frac{p}{2}\geq 1$.
One can still take $\theta=1$.

{\bf Case I.3: $g(\delta,\gamma)<p<\frac{9-6\delta}{4-2\delta}$.} Here
\beno
g(\delta,\gamma):= \frac{(9-6\delta)\delta}{(2-2\delta)[2\delta-\gamma (\frac32-\delta)]+2\delta}=\frac{3\delta}{2\delta-(1-\delta)\gamma}.
\eeno
One can verify that for  $0<\gamma< \frac{2\delta}{3-2\delta}$ there holds
\beno
g(\delta,\gamma)<\frac{9-6\delta}{4-2\delta}
\eeno
and
\beno
g(\delta,\gamma)>\frac{3}{2+\gamma}
\eeno
for $\delta,\gamma\in (0,\frac12)$. Hence, for
\beno
\frac{3\delta}{2\delta-(1-\delta)\gamma}\doteq g(\delta,\gamma)<p<\frac{9-6\delta}{4-2\delta}
\eeno
Then
(\ref{eq:D2}) becomes
\beno
\theta\leq \frac{(2-2\delta)p}{9-6\delta-2p}<1
\eeno
Furthermore,  there holds \beno
\frac{\delta}{[2\delta-\gamma(\frac32-\delta)]}<\frac{(2-2\delta)p}{9-6\delta-2p}
\eeno
 provided that
\beno
p> g(\delta,\gamma).
\eeno
This and (\ref{eq:gamma bound}) imply
\beno
\frac{\delta}{[2\delta-\gamma(\frac32-\delta)]}\leq \theta\leq \frac{(2-2\delta)p}{9-6\delta-2p}
\eeno
is reasonable, that is the set of $\theta$ is not empty. Further notice that since $p\leq \frac{5}{2}-\delta$ in this case, we have
$$\frac{(2-2\delta)p}{9-6\delta-2p}\leq \min\{1,\frac{p}{2},\frac{r}{2}\}.
$$

Thus one can choose $\theta$ as follows:
\beno
\theta=\frac12\{\frac{\delta}{[2\delta-\gamma (\frac32-\delta)]}+  \frac{(2-2\delta)p}{9-6\delta-2p}  \}\in(\frac12,1).
\eeno

{\bf \underline{Case II: $1<r< 2$.}}
 At this time $\frac{3}{1+\gamma}<p<\frac{3}{\gamma}$. Then  (\ref{eq:gamma bound}) becomes
\beno
\frac12<\frac{\delta}{[2\delta-\gamma(\frac32-\delta)]}\leq\theta\leq   \frac{r}{2},\quad  0<\gamma\leq \frac{2\delta}{3-2\delta}.
\eeno
This is equivalent to
\beno
\frac{\delta}{[2\delta-\gamma(\frac32-\delta)]}\leq \theta\leq \frac{1}{2+\gamma-\frac3p}\quad\textrm{and}\quad 0<\gamma\leq \frac{2\delta}{3-2\delta},
\eeno
where the choice of $\theta$ is not empty provided that $p\leq \frac{2\delta}{\gamma}$. Then for $\frac{3}{1+\gamma}<p\leq\frac{2\delta}{\gamma}$, we have
\ben\label{eq:gamma bound2}
\frac12<\frac{\delta}{[2\delta-\gamma(\frac32-\delta)]}\leq\theta\leq   \frac{r}{2}<1,\quad  0<\gamma< \frac{2\delta}{3-2\delta}.
\een

Moreover,
 (\ref{eq:D1}) holds for any  $\frac{3}{1+\gamma}<p\leq\frac{2\delta}{\gamma}$. Indeed, it is enough that (\ref{eq:D1}) holds for $p=\frac{2\delta}{\gamma}$:
\beno
(2-\frac{\gamma}{\delta}(\frac32-\delta))\frac{\delta}{[2\delta-\gamma(\frac32-\delta)]} \leq1
\eeno
which is obvious.

Thus for $\frac{3}{1+\gamma}<p\leq\frac{2\delta}{\gamma}$, it suffices to consider the conditions of (\ref{eq:gamma bound2}) and (\ref{eq:D2}).

In Case II below, we take any fixed $p\in(\frac{3}{1+\gamma},\frac{2\delta}{\gamma}]$ and $\delta\in (\frac{3\gamma}{2+2\gamma},\frac{1}{2})$. We show that (\ref{eq:D2})-(\ref{eq:gamma bound2}) can be satisfied with an appropriate $\theta$. We first consider $\delta\in [\frac{9\gamma}{4+6\gamma},\frac{1}{2})$ in Case II.1-II.2. In Case II.3, we finally treat $\delta\in (\frac{3\gamma}{2+2\gamma},\frac{9\gamma}{4+6\gamma})$.

{\bf Case II.1: $\frac92-3\delta\leq p\leq\frac{2\delta}{\gamma} $ with $\delta\in [\frac{9\gamma}{4+6\gamma},\frac{1}{2})$.} At this time
\beno
0<\gamma\leq \frac{2\delta}{\frac92-3\delta}<\frac{2\delta}{3-2\delta}.
\eeno
Moreover, (\ref{eq:D2}) holds. It is enough that
\beno
\frac12<\frac{\delta}{[2\delta-\gamma (\frac32-\delta)]}\leq \theta\leq \frac{r}{2}.
\eeno
Then one can take
\beno
\theta=\frac12\left\{\frac{\delta}{[2\delta-\gamma (\frac32-\delta)]}+\frac{r}{2} \right\}\in(\frac12,1).
\eeno

{\bf Case II.2: $\frac{3}{1+\gamma}<p<\frac92-3\delta$ with $\delta\in [\frac{9\gamma}{4+6\gamma},\frac{1}{2})$.} Note that this case is only non-empty for $\delta\in [\frac{9\gamma}{4+6\gamma},\frac{3\gamma+1}{2+2\gamma})$.

 We see that (\ref{eq:D2}) implies
\beno
\theta\leq \frac{(2-2\delta)p}{9-6\delta-2p}.
\eeno
 Using similar arguments as in {\bf Case I.3}, we see that this is compatible with (\ref{eq:gamma bound2}) if
\beno
\frac{\delta}{[2\delta-\gamma(\frac32-\delta)]}<\frac{(2-2\delta)p}{9-6\delta-2p},
\eeno
which holds if
\beno
p>\frac{3\delta}{2\delta-(1-\delta)\gamma}.
\eeno
It's easy to verify that
\beno
\frac{3\delta}{2\delta-(1-\delta)\gamma}<\frac{3}{1+\gamma}
\eeno
using
\beno
\gamma<\frac{2\delta}{3-2\delta}.
\eeno
Thus for $\frac{3}{1+\gamma}<p<\frac92-3\delta$ with $\delta\in [\frac{9\gamma}{4+6\gamma},\frac{1}{2})$, one can
take
\beno
\theta=\frac12\left\{\frac{\delta}{[2\delta-\gamma (\frac32-\delta)]}+\min\{\frac{r}{2},\frac{(2-2\delta)p}{9-6\delta-2p}  \}\right\}\in(\frac12,1).
\eeno

{\bf Case II.3: $\frac{3}{1+\gamma}< p\leq\frac{2\delta}{\gamma} $ with $\delta\in (\frac{3\gamma}{2+2\gamma},\frac{9\gamma}{4+6\gamma})$.} At this time $\frac{2\delta}{\gamma}\leq \frac92-3\delta $ and   (\ref{eq:D2})-(\ref{eq:gamma bound2}) implies
\beno
\frac{\delta}{[2\delta-\gamma(\frac32-\delta)]}\leq \theta\leq \frac{(2-2\delta)p}{9-6\delta-2p},
\eeno
which is meaningful if
\beno
p>\frac{3\delta}{2\delta-(1-\delta)\gamma}.
\eeno
Note that
\beno
\frac{3\delta}{2\delta-(1-\delta)\gamma}<\frac{3}{1+\gamma}
\eeno
due to
$\gamma<\frac{2\delta}{3-2\delta}<\delta.$
Thus for $\frac{3}{1+\gamma}<p\leq \frac{2\delta}{\gamma}$ with $\delta\in (\frac{3\gamma}{2+2\gamma},\frac{9\gamma}{4+6\gamma})$, one can
take
\beno
\theta=\frac12\left\{\frac{\delta}{[2\delta-\gamma (\frac32-\delta)]}+\min\{\frac{r}{2},\frac{(2-2\delta)p}{9-6\delta-2p}  \}\right\}\in(\frac12,1).
\eeno

The proof is complete.

\section{Proof of the $\varepsilon$- regularity criterion (Proposition \ref{Bqsmallreg})}
In this section, it is necessary to introduce the following notation for $z_0=(x_0,t_0)$, $v:Q(z_0,R)\rightarrow\mathbb{R}^3$ and $\pi:Q(z_0,R)\rightarrow \mathbb{R}$. We define (for $r\in (0,R)$):
\begin{equation}\label{Adef}
A(v,z_0,r):=\esssup_{t_0-r^2\leq t\leq t_{0}}\frac{1}{r}\int\limits_{B(x_0,r)} |v(x,t)|^2 dx ,
\end{equation}
\begin{equation}\label{Edef}
E(v,z_0,r):= \frac{1}{r}\int\limits_{Q(z_0,r)} |\nabla v|^2 dxds
,
\end{equation}
\begin{equation}\label{Cdef}
C(v,z_0,r):=\frac{1}{r^2}\int\limits_{Q(z_0,r)} |v|^3 dxds
,
\end{equation}
\begin{equation}\label{Sdef}
S(v,z_0,r):=\frac{1}{r^2}\int\limits_{Q(z_0,r)} ||v|^2-[|v|^2]_{B(x_0,r)}||v|dxds
,
\end{equation}
\begin{equation}\label{Ddef}
D(\pi, z_0,r):=\frac{1}{r^2}\int\limits_{Q(z_0,r)} |\pi-[\pi]_{B(x_0,r)}|^{\frac{3}{2}} dxds.
\end{equation}
Note that in \eqref{Sdef}-\eqref{Ddef}, we use the notation
$$[f]_{B(x_0,r)}=\frac{1}{\mu(B(x_0,r))}\int\limits_{B(x_0,r)} f dx, $$
where $\mu(B(x_0,r))$ denotes the Lebesgue measure of $B(x_0,r)$.
For $q\in (2,3)$ we also introduce the notation
\begin{equation}\label{Bqdef}
B_{q}(v,z_0,r):=\frac{1}{r^{3-q}}\int\limits_{Q(z_0,r)} |\nabla v|^2|v|^{q-2} dxds.
\end{equation}
In the above, when $z_0=(0,0)$ we will write $A(v,r)$, $B(v,r)$ etc, instead of $A(v,0,r)$ and $B(v,0,r)$.

The goal of this section is to prove Proposition \ref{Bqsmallreg} (stated in the Introduction).
In proving Proposition \ref{Bqsmallreg}, we will utilize the following Lemma, which is stated and proven in \cite{sereginsverakpressure} (specifically Lemma 3.3 in \cite{sereginsverakpressure}).
\begin{lemma}\label{smallkineticimpliesreg}
There exists a positive universal constant $\varepsilon_{*}$ such that the following holds true.

Suppose that $(v,\pi)$ is a suitable weak solution to the Navier-Stokes equations in $Q(0,1)$.
Furthermore, suppose that there exists an $R_{*}\in(0,1)$ such that
\begin{equation}\label{Asmall}
\sup_{0<R<R_{*}} A(v,R)<\varepsilon_{*}.
\end{equation}
Then the above assumptions imply that $(x,t)=(0,0)$ is a regular point of $v$.

In particular, there exists an $r\in (0,1)$ such that
$v\in L^{\infty}(Q(0,r)).$
\end{lemma}
In order to prove Proposition \ref{Bqsmallreg}, we must first prove three preliminary estimates, which we state as separate Lemmas.

\begin{lemma}\label{Sestlem}
For every $q\in (2,3)$, there exists $c^{(0)}_{q}>0$ such that the following holds true.

Suppose that $(v,\pi)$ is a suitable weak solution to the Navier-Stokes equations in $Q(0,1)$.
Furthermore, suppose that
\begin{equation}\label{Bqfinite2}
 B_{q}(v,1)<\infty.
\end{equation}
Then for every $0<r\leq 1$, we have the estimate
\begin{equation}\label{Sest}
S(v,r)\leq c^{(0)}_{q} (A(v,r))^{\frac{4-q}{4}} (C(v,r))^{\frac{1}{3}}(B_{q}(v,r))^{\frac{1}{2}}.
\end{equation}
\end{lemma}
\begin{proof}
By H\"{o}lder's inequality in space, we have
$$S(v,r)\leq \frac{1}{r^2}\int\limits^{0}_{-r^2} ds\Big(\int\limits_{B(0,r)}||v|^2-[|v|^2]_{B(0,r)}|^{\frac{3}{2}} dx\Big)^{\frac{2}{3}}\Big(\int\limits_{B(0,r)}|v|^3 dx\Big)^{\frac{1}{3}}. $$
Next, using the Poincar\'{e}-Sobolev inequality, we infer that
\begin{equation}\label{Sfirstest}
S(v,r)\leq \frac{C_{univ,(5)}}{r^2}\int\limits^{0}_{-r^2} ds\Big(\int\limits_{B(0,r)}|\nabla v||v| dx\Big)\Big(\int\limits_{B(0,r)}|v|^3 dx\Big)^{\frac{1}{3}}.
\end{equation}
Next, note that for almost every $s\in(-r^2,0)$ we can write
$$\int\limits_{B(0,r)} |\nabla v(x,s)||v(x,s)| dx=\int\limits_{B(0,r)} |\nabla v(x,s)||v(x,s)|^{\frac{q-2}{2}}|v(x,s)|^{\frac{4-q}{2}} dx. $$
Recalling that $q\in (2,3)$ and applying H\"{o}lder's inequality twice thus gives
\begin{equation}\label{vnablavkeyest}
\int\limits_{B(0,r)} |\nabla v(x,s)||v(x,s)| dx\leq C_{q,(1)}r^{\frac{3(q-2)}{4}}\Big(\int\limits_{B(0,r)} |v(x,s)|^2 dx\Big)^{\frac{4-q}{4}}\Big(\int\limits_{B(0,r)}|\nabla v(x,s)|^2 |v(x,s)|^{q-2} dx\Big)^{\frac{1}{2}}.
\end{equation}
Inserting this into \eqref{Sfirstest} gives
\begin{equation}\label{Ssecondest}
S(v,r)\leq C_{q,(2)}r^{\frac{q-5}{2}}(A(v,r))^{\frac{4-q}{4}}\int\limits^{0}_{-r^2} ds\Big(\int\limits_{B(0,r)}|\nabla v|^2|v|^{q-2} dx\Big)^{\frac{1}{2}}\Big(\int\limits_{B(0,r)}|v|^3 dx\Big)^{\frac{1}{3}}.
\end{equation}
Applying H\"{o}lder's inequality to \eqref{Ssecondest} (twice in the time variable) readily gives the desired conclusion.
\end{proof}
\begin{lemma}\label{lemCest}
For every $q\in (2,3)$, there exists $c^{(1)}_{q}>0$ such that the following holds true.

Suppose that $(v,\pi)$ is a suitable weak solution to the Navier-Stokes equations in $Q(0,1)$.
Furthermore, suppose that
\begin{equation}\label{Bqfinite1}
 B_{q}(v,1)<\infty.
\end{equation}
Then for every $0<r\leq\rho<1$, we have the estimate
\begin{equation}\label{Cest}
C(v,r)\leq c^{(1)}_{q}\Big\{\Big(\frac{r}{\rho}\Big)^3(A(v,\rho))^{\frac{3}{2}}+(A(v,\rho))^{\frac{3(4-q)}{8}}\Big(\frac{\rho}{r}\Big)^{3}(B_{q}(v,\rho))^{\frac{3}{4}}\Big\}.
\end{equation}
\end{lemma}
\begin{proof}
Let $0<r\leq\rho<1$. From p.g 800 of \cite{CKN}, we see that for almost every $s\in (-r^2,0)$ we have
$$\int\limits_{B(0,r)} |v(x,s)|^2 dx\leq C_{univ,(6)}\rho\int\limits_{B(0,\rho)} |\nabla v(x,s)||v(x,s)| dx+C_{univ,(6)}\rho\Big(\frac{r}{\rho}\Big)^3 A(v,\rho). $$
Applying \eqref{vnablavkeyest} (with $\rho$ replacing $r$) gives that for  almost every $s\in (-r^2,0)$:
\begin{equation}\label{vL2est}
\int\limits_{B(0,r)} |v(x,s)|^2 dx\leq C_{q,(1)}\rho^{\frac{q+1}{2}} A(v,\rho)^{\frac{4-q}{4}}\Big(\int\limits_{B(0,\rho)} |\nabla v(x,s)|^2 |v(x,s)|^{q-2} dx\Big)^{\frac{1}{2}}+C_{univ,(6)}\rho\Big(\frac{r}{\rho}\Big)^3 A(v,\rho).
\end{equation}
Next, we write
$$C(v,r)=\frac{1}{r^2}\int\limits^{0}_{-r^2}\int\limits_{B(0,r)} (|v(x,s)|^2-[|v|^2]_{B(0,r)})|v(x,s)| dxds+\frac{1}{r^2}\int\limits^{0}_{-r^2}\int\limits_{B(0,r)} [|v|^2]_{B(0,r)}|v(x,s)| dxds. $$
Hence, using H\"{o}lder's inequality we have
\begin{equation}\label{CinequalityS}
C(v,r)\leq S(v,r)+\frac{C_{univ,(7)}}{r^2}\int\limits_{-r^2}^{0}\frac{1}{r^{\frac{3}{2}}}\Big(\int\limits_{B(0,r)} |v(x,s)|^2 dx\Big)^{\frac{3}{2}} ds.
\end{equation}
Now, using Lemma \ref{Sestlem} we have
\begin{equation}\label{Cfirstinequality}
C(v,r)\leq c^{(0)}_{q} (A(v,r))^{\frac{4-q}{4}} (C(v,r))^{\frac{1}{3}}(B_{q}(v,r))^{\frac{1}{2}}+\frac{C_{univ,(7)}}{r^2}\int\limits_{-r^2}^{0}\frac{1}{r^{\frac{3}{2}}}\Big(\int\limits_{B(0,r)} |v(x,s)|^2 dx\Big)^{\frac{3}{2}} ds.
\end{equation}
Using Young's inequality, we then infer that
\begin{equation}\label{Cyoungsinequality}
C(v,r)\leq C_{q,(3)}(A(v,r))^{\frac{3(4-q)}{8}} (B_{q}(v,r))^{\frac{3}{4}}+\frac{C_{univ,(8)}}{r^2}\int\limits_{-r^2}^{0}\frac{1}{r^{\frac{3}{2}}}\Big(\int\limits_{B(0,r)} |v(x,s)|^2 dx\Big)^{\frac{3}{2}} ds.
\end{equation}
Using this, together with \eqref{vL2est} and H\"{o}lder's inequality in the time variable, one deduces that
\begin{align}\label{Cinequalityrrho}
\begin{split}
&C(v,r)\leq C_{q,(3)}(A(v,r))^{\frac{3(4-q)}{8}} (B_{q}(v,r))^{\frac{3}{4}}+C_{univ,(9)}\Big(\frac{r}{\rho}\Big)^{3} (A(v,\rho))^{\frac{3}{2}}
\\&+C_{univ,(9)}(A(v,\rho))^{\frac{3(4-q)}{8}}\Big(\frac{\rho}{r}\Big)^{3}(B_{q}(v,\rho))^{\frac{3}{4}}
\\
&\leq \Big(C_{q,(3)}\Big(\frac{\rho}{r}\Big)^{\frac{3(4-q)}{8}+\frac{3(3-q)}{4}}+C_{univ,(9)}\Big(\frac{\rho}{r}\Big)^{3}\Big)(A(v,\rho))^{\frac{3(4-q)}{8}}(B_{q}(v,\rho))^{\frac{3}{4}}+C_{univ,(9)}\Big(\frac{r}{\rho}\Big)^{3} (A(v,\rho))^{\frac{3}{2}}.
\end{split}
\end{align}
Using this, along with the fact that $q\in (2,3)$, the desired conclusion readily follows.
\end{proof}

\begin{lemma}\label{Destlem}
For every $q\in (2,3)$, there exists a positive constant $c^{(2)}_{q}$ (depending only on $q$) such that the following holds true.
Suppose that $(v,\pi)$ is a suitable weak solution to the Navier-Stokes equations in $Q(0,1)$.
Furthermore, suppose that
\begin{equation}\label{Bqfinite3}
 B_{q}(v,1)<\infty.
\end{equation}
Then for every $0<r\leq\rho<1$, we have the estimate
\begin{equation}\label{Dest}
D(\pi,r)\leq c^{(2)}_{q}\Big\{\Big(\frac{\rho}{r}\Big)^2(A(v,\rho))^{\frac{3(4-q)}{8}}(B_{q}(v,\rho))^{\frac{3}{4}}+\Big(\frac{r}{\rho}\Big)^{\frac{5}{2}} D(\pi,\rho)\Big\}.
\end{equation}
\end{lemma}
\begin{proof}
Let $\chi_{B(0,\rho)}:\mathbb{R}^3\rightarrow \{0,1\}$ denote the indicator function, that is equal to 1 on $B(0,\rho)$ and is zero elsewhere.

For almost every $s\in (-\rho^2,0)$, we write
\begin{equation}\label{pressuredecomp}
\pi(x,s)=\pi_{1}(x,s)+\pi_{2}(x,s)\,\,\,\textrm{with}\,\,\,\pi_{1}:= \mathcal{R}_{i}\mathcal{R}_{j}(\chi_{B(0,\rho)}(v_{i}v_{j}-[v_{i}v_{j}]_{B(0,\rho)})).
\end{equation}
In \eqref{pressuredecomp}, $R=(R_{\alpha})_{\alpha=1,\ldots 3}$ is the Riesz transform and we utilize the Einstein summation convention.

It is not difficult to show that for almost every $s\in (-\rho^2,0)$, we must have
\begin{equation}\label{q2harmonic}
\Delta \pi_{2}(x,s)=0\,\,\,\textrm{in}\,\,\,B(0,\rho).
\end{equation}
Using properties of harmonic functions and the same arguments as in \cite{gregory2014lecture} (specifically p.g 145 of \cite{gregory2014lecture}) yields the following estimate:
\begin{equation}\label{pressureestharmonicity}
D(\pi,r)\leq\frac{C_{univ,(10)}}{r^2}\int\limits_{Q(0,r)} |\pi_{1}(x,s)|^{\frac{3}{2}} dxds+C_{univ,(10)}\Big(\frac{r}{\rho}\Big)^{\frac{5}{2}}\Big(D(\pi,\rho)+\frac{1}{{\rho}^2}\int\limits_{Q(0,\rho)} |\pi_{1}(x,s)|^{\frac{3}{2}} dxds\Big).
\end{equation}
It remains to estimate $$\frac{1}{{\rho}^2}\int\limits_{Q(0,\rho)} |\pi_{1}(x,s)|^{\frac{3}{2}} dxds.$$
Recalling \eqref{pressuredecomp} and applying Calder\'{o}n-Zygmund estimates give that for almost every $s\in (-\rho^2,0)$
$$\int\limits_{B(0,\rho)} |\pi_{1}(x,s)|^{\frac{3}{2}} dx\leq C_{univ,(11)}\sum_{i,j=1}^{3}\int\limits_{B(0,\rho)} |v_{i}v_{j}(x,s)-[v_{i}v_{j}(\cdot,s)]_{B(0,\rho)}|^{\frac{3}{2}} dx. $$
Next, the Poincar\'{e}-Sobolev inequality gives that for almost every $s\in (-\rho^2,0)$:
\begin{equation}\label{p1est1}
\int\limits_{B(0,\rho)} |\pi_{1}(x,s)|^{\frac{3}{2}} dx\leq C_{univ,(12)}\Big(\int\limits_{B(0,\rho)} |v(x,s)||\nabla v(x,s)| dx\Big)^{\frac{3}{2}}.
\end{equation}
Next, utilizing \eqref{vnablavkeyest} (with $\rho$ instead of $r$) yields
\begin{equation}\label{p1est2}
\int\limits_{B(0,\rho)} |\pi_{1}(x,s)|^{\frac{3}{2}} dx\leq C_{q,(4)}\rho^{\frac{3(q-1)}{4}}(A(v,\rho))^{\frac{3(4-q)}{8}}\Big(\int\limits_{B(0,\rho)}|\nabla v(x,s)|^2|v(x,s)|^{q-2} dx\Big)^{\frac{3}{4}}.
\end{equation}
Integrating this in time and applying H\"{o}lder's inequality in time yields
\begin{equation}\label{p1keyest}
\frac{1}{\rho^2}\int\limits_{Q(0,\rho)} |\pi_{1}|^{\frac{3}{2}} dxds\leq C_{univ,(13)}(A(v,\rho))^{\frac{3(4-q)}{8}}(B_{q}(v,\rho))^{\frac{3}{4}}.
\end{equation}
Combining this with \eqref{pressureestharmonicity} then gives the desired conclusion.
\end{proof}
\subsection{Proof of Proposition \ref{Bqsmallreg}}
\begin{proof}
First, recall that the local energy inequality implies that for all positive functions $\varphi\in C^{\infty}_{0}(B(0,1)\times (-1,\infty))$ the following inequality holds for almost every $t\in (-1,0)$:
\begin{align}\label{localnenergyequality}
\begin{split}
&\int\limits_{B(0,1)} |v(x,t)|^2 \varphi(x,t) dx+2\int\limits_{-1}^{t}\int\limits_{B(0,1)} |\nabla v(x,s)|^2\varphi(x,s) dxds\leq \int\limits_{-1}^{t}\int\limits_{B(0,1)} | v(x,s)|^2(\partial_{s}+\Delta)\varphi(x,s) dxds+\\&
+\int\limits_{-1}^{t}\int\limits_{B(0,1)} (| v(x,s)|^2-[|v(\cdot,s)|^2]_{B(0,2r)})v(x,s)\cdot\nabla\varphi+ 2(\pi(x,s)-[\pi(\cdot,s)]_{B(0,2r)})v(x,s)\cdot\nabla\varphi dxds.
\end{split}
\end{align}
Here, we have taken $0\leq r\leq \frac{1}{2}.$
By selecting an appropriate test function $\varphi$, together with making use of H\"{o}lder's inequality, one can show that \eqref{localnenergyequality} implies the following. Namely, for any $0<r\leq \frac{1}{2}$:
\begin{equation}\label{localenergybasicest}
A(v,r)+E(v,r)\leq C_{univ,(14)}\{(C(v,2r))^{\frac{2}{3}}+S(v,2r)+(D(\pi,2r))^{\frac{2}{3}}(C(v,2r))^{\frac{1}{3}}\}.
\end{equation}
Recall that in Proposition \ref{Bqsmallreg}, we assume that
\begin{equation}\label{Bqsmallrecall}
\sup_{0<\rho<1} B_{q}(v,\rho)<\varepsilon,
\end{equation}
where $\varepsilon=\varepsilon_{q}\in (0,1)$ is to be determined.

Inserting \eqref{Bqsmallrecall} into Lemmas \ref{Sestlem}-\ref{Destlem} allows us to infer that the following hold true. In particular, for all $0<2r\leq\rho<1$ we have
\begin{equation}\label{energyestBqsmall}
A(v,r)+E(v,r)\leq C_{q,(5)}\{C(v,2r))^{\frac{2}{3}}+(D(\pi,2r))^{\frac{2}{3}}(C(v,2r))^{\frac{1}{3}}+(A(v,2r))^{\frac{4-q}{4}}(C(v,2r))^{\frac{1}{3}}\varepsilon^{\frac{1}{2}}\},
\end{equation}
\begin{equation}\label{CestBqsmall}
C(v,2r)\leq C_{q,(6)}\Big\{\Big(\frac{r}{\rho}\Big)^3 (A(v,\rho))^{\frac{3}{2}}+(A(v,\rho))^{\frac{3(4-q)}{8}}\Big(\frac{\rho}{r}\Big)^{3}\varepsilon^{\frac{3}{4}}\Big\}\,\,\,\textrm{and}
\end{equation}
\begin{equation}\label{DestBqsmall}
D(\pi,2r)\leq C_{univ,(15)}\Big\{\Big(\frac{\rho}{r}\Big)^2(A(v,\rho))^{\frac{3(4-q)}{8}}\varepsilon^{\frac{3}{4}}+\Big(\frac{r}{\rho}\Big)^{\frac{5}{2}} D(\pi,\rho)\Big\}.
\end{equation}
Let us define
\begin{equation}\label{Ebigdef}
\mathcal{E}(v,\pi,r):=(A(v,r))^{\frac{3}{2}}+(D(\pi,r))^{2}.
\end{equation}
From \eqref{energyestBqsmall} and Young's inequality, we see that
\begin{equation}\label{Ebigfirstest}
\mathcal{E}(v,\pi,r)\leq C_{q,(7)}\{C(v,2r)+(D(\pi,2r))^2+(A(v,2r))^{\frac{3(4-q)}{4}}\varepsilon^{\frac{3}{2}}\}.
\end{equation}
Substituting \eqref{CestBqsmall}-\eqref{DestBqsmall} into \eqref{Ebigfirstest} (and recalling $2r\leq\rho$) gives
\begin{equation}\label{Ebigsecondest}
\mathcal{E}(v,\pi,r)\leq C_{q,(8)}\Big\{\Big(\frac{r}{\rho}\Big)^{3}\mathcal{E}(v,\pi,\rho)+(A(v,\rho))^{\frac{3(4-q)}{8}}\Big(\frac{\rho}{r}\Big)^{3}\varepsilon^{\frac{3}{4}}+\Big(\frac{\rho}{r}\Big)^4(A(v,\rho))^{\frac{3(4-q)}{4}}\varepsilon^{\frac{3}{2}}+(A(v,2r))^{\frac{3(4-q)}{4}}\varepsilon^{\frac{3}{2}}\Big\}.
\end{equation}
Using $A(v,2r)\leq \frac{\rho}{2r} A(v,\rho)$, $q\in (2,3)$, $\varepsilon\in (0,1)$ and Young's inequality, we can now deduce the following. Namely,
\begin{equation}\label{Ebigthirdest}
\mathcal{E}(v,\pi,r)\leq C_{q,(9)}\Big\{\Big(\frac{r}{\rho}\Big)^{3}\mathcal{E}(v,\pi,\rho)+\Big(\frac{\rho}{r}\Big)^6(A(v,\rho))^{\frac{3(4-q)}{4}}\varepsilon^{\frac{1}{2}}+\varepsilon\Big\}.
\end{equation}
Noting that $q\in (2,3)$ implies $\frac{2}{4-q}\in (1,2)$, we can apply Young's inequality once more.
Thus for fixed positive $\delta$  and for all $0<2r\leq\rho\leq 1$, we get
\begin{equation}\label{Ebigmainestgeneral}
\mathcal{E}(v,\pi,r)\leq C_{q,(10)}\Big\{\Big(\Big(\frac{r}{\rho}\Big)^{3}+\delta\Big)\mathcal{E}(v,\pi,\rho)+\varepsilon+\delta^{-\frac{4-q}{q-2}}\varepsilon^{\frac{1}{q-2}}\Big(\frac{\rho}{r}\Big)^{\frac{12}{q-2}}\Big\}.
\end{equation}
Therefore, for any $0<\theta\leq\tfrac{1}{2}$ and $0<\rho\leq 1$ we have
\begin{equation}\label{Ebigthetarhofirstest}
\mathcal{E}(v,\pi,\theta\rho)\leq C_{q,(11)}\{(\theta^3+\delta)\mathcal{E}(v,\pi,\rho)+\varepsilon+\delta^{-\frac{4-q}{q-2}}\varepsilon^{\frac{1}{q-2}}\theta^{-\frac{12}{q-2}}\}.
\end{equation}
Next we fix $\theta=\theta_{q}\in (0,\frac{1}{2}]$ and $\delta=\delta_{q}>0$ (independent of $\varepsilon$) such that
\begin{equation}\label{thetadeltafix}
2C_{q,(11)}\theta\leq 1\,\,\,\textrm{and}\,\,\,C_{q,(11)}\delta<\frac{\theta^2}{2}.
\end{equation}
This gives
\begin{equation}\label{Ebigthetarhosecondest}
\mathcal{E}(v,\pi,\theta\rho)\leq \theta^2 \mathcal{E}(v,\pi,\rho)+ G_{\delta,\theta}(\varepsilon)\,\,\,\textrm{with}\,\,\, G_{\delta,\theta}(\varepsilon):=C_{q,(11)}(\varepsilon+\delta^{-\frac{4-q}{q-2}}\varepsilon^{\frac{1}{q-2}}\theta^{-\frac{12}{q-2}}).
\end{equation}
Setting $\rho=1$ and iterating \eqref{Ebigthetarhosecondest} gives that for every $k\in\mathbb{N}$:
\begin{equation}\label{Ebigiterated}
\mathcal{E}(v,\pi,\theta^k)\leq \theta^{2k}\mathcal{E}(v,\pi,1)+\frac{G_{\delta,\theta}(\varepsilon)}{1-\theta^2}.
\end{equation}
Thus, for all $r\in (0,1]$ we obtain
\begin{equation}\label{Ebigrdecay}
\mathcal{E}(v,\pi,r)\leq \frac{r^2}{\theta^6} \mathcal{E}(v,\pi,1)+\frac{G_{\delta,\theta}(\varepsilon)}{\theta^4(1-\theta^2)}.
\end{equation}
From \eqref{Ebigthetarhosecondest} and since $\delta$ and $\theta$ are fixed independently of $\varepsilon$, we see that $\lim_{\varepsilon\rightarrow 0^+} G_{\delta,\theta}(\varepsilon)=0.$ So, there exists $\varepsilon=\varepsilon_{q}$ such that
$$\frac{G_{\delta,\theta}(\varepsilon)}{\theta^4(1-\theta^2)}<\Big(\frac{\varepsilon_{*}}{2}\Big)^{\frac{3}{2}}. $$
Here, $\varepsilon_{*}$ is as in Lemma \ref{smallkineticimpliesreg}.
Hence, one concludes
\begin{equation}\label{Alimsupsmall}
\limsup_{R_{*}\downarrow 0} \sup_{0<R<R_{*}} A(v,R)<\varepsilon_{*}.
\end{equation}
Hence, by Lemma \ref{smallkineticimpliesreg}, we see that $(x,t)=(0,0)$ is a regular point of $v$ as required.
\end{proof}
\section{Proof of the reduction in the Hausdorff dimension of the singular set (Theorems \ref{hausdorffdimreduce}-\ref{hausdorffdimreducegen2})}
Having established the $\varepsilon$- regularity criterion (Proposition \ref{Bqsmallreg}), in this section we prove Theorems \ref{hausdorffdimreduce}-\ref{hausdorffdimreducegen2} by means of Theorem \ref{thm:higherinteggen} and a covering argument found in Caffarelli, Kohn and Nirenberg's paper \cite{CKN}.
\subsection{Proof of Theorem \ref{hausdorffdimreduce}}
\begin{proof}
Recall that the assumptions in Theorem  \ref{hausdorffdimreduce} allow us to apply Theorem \ref{thm:higherinteggen}. In particular, for $(p,r,\delta,\gamma)$ as in Theorem \ref{hausdorffdimreduce} and $\pi\in L^{r}_{t}L^{p}_{x}(\mathbb{R}^3\times(0,T^*))$ we have:
\begin{equation}\label{highintegrecall}
\int\limits_{t_{1}}^{T^*}\int\limits_{\mathbb{R}^3} |v|^{q-2}|\nabla v|^2 dxds<\infty\quad\textrm{with}\quad q=3-2\delta.
\end{equation}
Let $(x,T^*)$ be any member of the singular set $\sigma$, where $\sigma$ is defined in \eqref{sigmadefhausdorff}. From Proposition \ref{Bqsmallreg}, we see that for all $\hat{\varepsilon}\in (0, \sqrt{T^*-t_{1}})$ there exists $r_{x,\hat{\varepsilon}}<\frac{\hat{\varepsilon}}{5}$ such that
\begin{equation}\label{Bqlowerboundsing}
\int\limits_{T^*-r_{x,\hat{\varepsilon}}^2}^{T^*}\int\limits_{B(x,r_{x,\hat{\varepsilon}})}  |v|^{q-2}|\nabla v|^2 dxds\geq \frac{\varepsilon_{q}}{2} r_{x,\hat{\varepsilon}}^{3-q}=\frac{\varepsilon_{q}}{2} r_{x,\hat{\varepsilon}}^{2\delta}.
\end{equation}
By Vitali's covering lemma, there exists a countable subset $\{x_{i}\}_{i\in\mathbb{N}}\subset\sigma$ such that
\begin{equation}\label{disjoint}
\bar{B}(x_{i}, r_{x_{i},\hat{\varepsilon}})\cap \bar{B}(x_{j}, r_{x_{j},\hat{\varepsilon}})=\emptyset\,\,\,\textrm{for}\,\,\,i\neq j\,\,\,\textrm{and}
\end{equation}
\begin{equation}\label{countablecovering}
\sigma\subseteq \cup_{x\in\sigma} \bar{B}(x,r_{x,\hat{\varepsilon}})\subseteq \cup_{i=1}^{\infty} \bar{B}(x_{i}, 5r_{x_{i},\hat{\varepsilon}}).
\end{equation}
From \eqref{Bqlowerboundsing}, we see that for every $i\in\mathbb{N}$
\begin{equation}\label{Bqlowerboundcountable}
\frac{\varepsilon_{q}}{2} (5r_{x_{i},\hat{\varepsilon}})^{2\delta}\leq 5^{2\delta} \int\limits_{T^*-\hat{\varepsilon}^2}^{T^*}\int\limits_{B(x_{i},r_{x_{i},\hat{\varepsilon}})}  |v|^{q-2}|\nabla v|^2 dxds.
\end{equation}
Using \eqref{disjoint}, we see that
\begin{align}\label{sumdisjoint}
\begin{split}
 \sum_{i=1}^{\infty}\frac{\varepsilon_{q}}{2}(5r_{x_{i},\hat{\varepsilon}})^{2\delta}&\leq 5^{2\delta}\sum_{i=1}^{\infty}\int\limits_{T^*-\hat{\varepsilon}^2}^{T^*}\int\limits_{B(x_{i},r_{x_{i},\hat{\varepsilon}})}  |v|^{q-2}|\nabla v|^2 dxds\\
 &\leq 5^{2\delta}\int\limits_{T^*-\hat{\varepsilon}^2}^{T^*}\int\limits_{\mathbb{R}^3}  |v|^{q-2}|\nabla v|^2 dxds.
\end{split}
\end{align}
Using this, \eqref{countablecovering}, the fact that $5r_{x_{i},\hat{\varepsilon}}<\hat{\varepsilon}$ and the definition of Hausdorff measure, we infer the following. Namely
$$\mathcal{H}^{2\delta,\hat{\varepsilon}}(\sigma)\leq \frac{ 5^{2\delta}2}{\varepsilon_{q}}\int\limits_{T^*-\hat{\varepsilon}^2}^{T^*}\int\limits_{\mathbb{R}^3}  |v|^{q-2}|\nabla v|^2 dxds. $$
Sending $\hat{\varepsilon}\downarrow 0$ and applying the dominated convergence theorem gives us the desired conclusion.
\end{proof}

\subsection{Proof of Theorem \ref{hausdorffdimreducegen2}}

\begin{proof}
{\bf \underline{Case I: $r> 2$.}} For any $\delta\in (0,\frac{1}{2})$, one can choose $\gamma_0>0$ such that for any $\gamma\in (0,\gamma_0)$ it satisfies
\beno
i)~~ \delta>\frac{3\gamma}{2+2\gamma}.
\eeno
Using Proposition \ref{interpolativepropertyLorentz}, we interpolate \eqref{typeIpresgenmaintheo2} with\footnote{Recall that for a weak Leray-Hopf solution $v:\mathbb{R}^3\times (0,\infty)\rightarrow \mathbb{R}^3$, it is known that the associated pressure $\pi$ belongs to $L^{a}_{t}L^{b}_{x}$ with $$\frac{2}{a}+\frac{3}{b}=3$$ and $1\leq a<\infty.$} $\pi\in L^{1}_{t}L^{3}_{x}(\mathbb{R}^3\times (0,T^*))$ gives $\pi\in L^{r_1}_{t}L^{s_1}_{x}(\mathbb{R}^3\times (0,T^*))$ with
\beno
ii)~~ \frac{2}{r_1}+\frac{3}{s_1}=2+\gamma,\quad 2\leq r_1<r,\frac32<s<s_1.
\eeno
Moreover, $$\frac{1}{s_1}=\frac{1}{s}+\gamma\Big(\frac{1}{3}-\frac{1}{s}\Big) .$$
Using $s>\frac{3}{2}$, one can choose
$$\gamma_0<\frac{\delta(2s-3)}{s-3\delta},$$ which ensures that
\beno
iii)~~g(\delta,\gamma)\doteq \frac{3\delta}{2\delta-(1-\delta)\gamma}<s_1.
\eeno

By Theorem
\ref{thm:higherinteggen}, we have
\beno
\int\limits_{t_1}^{T^*}\int\limits_{\mathbb{R}^3} |\nabla v|^2|v|^{1-2\delta} dxdt<\infty
\eeno
which implies $$\mathcal{H}^{2\delta}(\sigma)=0$$ due to Proposition \ref{Bqsmallreg} and the same arguments as the proof of Theorem  \ref{hausdorffdimreduce}.

{\bf \underline{Case II: $1< r< 2$.}} In this case, $s\in (3,\infty)$. Hence, for any $\delta \in(0,\frac{1}{2})$, one can choose $0<\gamma_1<\frac{1}{3}$ such that for any $\gamma\in (0,\gamma_1)$ the following inequalities hold:
\begin{equation}\label{deltacase2}
\delta\geq \frac{\gamma}{2\Big(\frac{1-\gamma}{s}+\frac{2\gamma}{3}\Big)}>\frac{3\gamma}{2+2\gamma}.
\end{equation}
Interpolating with $\pi\in L^{2}_{t}L^{\frac{3}{2}}_{x}(\mathbb{R}^3\times (0,T^*))$ gives $\pi\in L^{r_1}_{t}L^{s_1}_{x}(\mathbb{R}^3\times (0,T^*))$ with
\beno
\frac{2}{r_1}+\frac{3}{s_1}=2+\gamma,\quad 1\leq r<r_1<2\quad\textrm{and}\quad\frac{3}{2}<s_1<s.
\eeno
Moreover,
$$\frac{1}{s_1}=\frac{1-\gamma}{s}+\frac{2\gamma}{3} .$$
Furthermore,  $s>3$ and \eqref{deltacase2} ensures that
\beno
\frac{3}{1+\gamma}< s_1\leq \frac{2\delta}{\gamma}.
\eeno
Then Theorem
\ref{thm:higherinteggen} applies similarly to \textbf{Case I} and $\mathcal{H}^{2\delta}(\sigma)=0$.

{\bf \underline{Case III: $r=2$ and $s=3$.}}
In this case, for any $\delta\in (0,\frac{1}{2})$, one can choose $0<\gamma_{2}<\frac{1}{3}$ such that for any $\gamma\in (0,\gamma_{2})$ the following inequalities hold:
\begin{equation}\label{deltacase3}
\delta\geq \frac{3\gamma}{2+\gamma}>\frac{3\gamma}{2+2\gamma}.
\end{equation}
Interpolating  with $\pi\in L^{\frac{4}{3}}_{t}L^{2}_{x}(\mathbb{R}^3\times (0,T^*))$ gives $\pi\in L^{r_1}_{t}L^{s_1}_{x}(\mathbb{R}^3\times (0,T^*))$ with
$$\frac{1}{s_1}=\frac{1-\gamma}{3}+\frac{\gamma}{2}=\frac{2+\gamma}{6}. $$
Using this and \eqref{deltacase3}, we readily see that
\beno
\frac{3}{1+\gamma}< s_1\leq \frac{2\delta}{\gamma}.
\eeno
The remaining arguments required are identical to the previous case and are hence omitted.

The proof is complete.
\end{proof}
\subsection*{Acknowledgement}
W. Wang was supported by NSFC under grant 12071054 and by Dalian High-level Talent Innovation Project (Grant 2020RD09).

\bibliography{refs}
\bibliographystyle{plain}

\end{document}